\documentclass{amsart}
\usepackage[utf8]{inputenc}
\usepackage[english]{babel}
\usepackage{amsthm}
\usepackage{amssymb}
\usepackage{comment}
\usepackage{setspace}
\usepackage{url}
\theoremstyle{definition}

\newtheorem{theorem}{Theorem}[section]
\newtheorem{lemma}[theorem]{lemma}
\newtheorem{corollary}[theorem]{corollary}
\theoremstyle{remark}

\title{Inequality and Nyman-Beurling-Baez-Duarte criteria}
\author{Kwok Kwan Wong}

\date{June 2022}

\begin{document}

\maketitle
\begin{abstract}
    We proposed a proof of the Riemann hypothesis. The proof is based on the Nyman-Beurling-Baez-Duarte condition. By proving existence of the solution for a system of inequalities, we can show that there is a sequence, which act as the coefficient of Beurling's sequence, can approximate the constant vector in a weighted Hilbert space. 
\end{abstract}
\section{Introduction}

\begin{comment}
The Riemann hypothesis was raised by Riemann in 1859 \cite{riemann}. The hypothesis is about the zeros of the Riemann-Zeta function $\zeta$, $\zeta$ has the trivial zero, which are negative even integers, and the nontrivial zeros. Riemann posed a hypothesis that the real part of the nontrivial zeros are $\frac{1}{2}$, which we call the \textit{Riemann hypothesis}\\
From 1859 to now, many scholars attempted to prove or disprove the Riemann hypothesis\cite{odlyzko,turan,haselgrove,borwein,spira,connes,fesenko,suzuki}, since the Riemann hypothesis had a big impact on the field of mathematics, such as the distributions of prime numbers \cite{koch}, the large prime gap conjecture \cite{cramer}, etc. 
\end{comment}
The Riemann hypothesis was raised by Riemann in 1859 \cite{riemann}. The hypothesis is about the zeros of the Riemann-Zeta function $\zeta$, $\zeta$ has the trivial zero, which are negative even integers, and the nontrivial zeros. Riemann posed a hypothesis that the real part of the nontrivial zeros are $\frac{1}{2}$, which we call the \textit{Riemann hypothesis}.\\

To prove or disprove the Riemann hypothesis, many scholars try to formulate the Riemann hypothesis in another way \cite{riesz,salem,speiser}. In particular, Nyman and Beurling show that the Riemann hypothesis is true if and only if the space of the Beurling function is dense in Hilbert space $L^{2}((0,1))$ \cite{nyman,beurling}. Baez-Duarte has restated and strengthened this condition to be the Riemann hypothesis is true if and only if the characteristic function $\chi_{(0,1]}$ belongs to the closure of the space of the \textit{natural Beurling function} in the Hilbert space $L^{2}((0,\infty))$ \cite{Luis}. Bagchi reformulates the condition to be if the constant sequence belongs to the closure of the span of the \textit{Beurling sequence} in the $l^{2}(\mathbb{N})$ with a weighted inner product \cite{bagchi}. \\
There are numerous working on this approach\cite{maier,yang,habsieger,bettin,bober,burnol}. Our contribution is that by show that for large enough $n$, we can bound the all component of Beurling sequence with any positive number. Normally to show a vector belongs to a subspace, one required to find the coefficients of the basis of the subspace. Examples is the natural approximation, we overcome this technical difficulty by only showing the coefficients exists without explicitly constructing them.  So we have the following theorem.
\begin{theorem}
    The Riemann hypothesis is true.
\end{theorem}
The details of the approach and the proof of this theorem are discuss in the next session.

\section{Our approach to the problem}
The Hilbert space we consider is $l^{2}(\mathbb{N}):=H$ over $\mathbb{C}$ with the norm induced by the inner product.
\begin{center}
    $\langle a,b\rangle=\sum_{n=1}^{\infty}\frac{a^{*}_{n}b_{n}}{n(n+1)}$.
\end{center}
Observe that bounded sequences belong to $H$ as well.\\

We adopt the notion in \cite{bagchi}, we introduce the sequence $\gamma_{l}=(\{\frac{n}{l}\})=(1/l,2/l,...)$ for $l\in\mathbb{N}$, where $\{x\}$ is the fractional part function. It is easy to see that $\gamma_{l}\in H$ for all $l$. Denote the $span(\gamma_{l},l\in\mathbb{N})=B$, we call the $B$ the space of the \textit{Beurling sequences}. Let $\gamma=(1,1,...)$ be the constant sequence, it is easy to see that it belongs to $H$. Denote $||.||_{H}$ the norm of $H$.
In \cite{bagchi}, it states the following theorem.
\begin{theorem}
The Riemann hypothesis is equivalent to $\gamma\in\overline{B}$, and is equivalent to $B$ is dense in $H$.
\end{theorem}
\begin{proof}
See the proof of Theorem 1 of \cite{bagchi}.
\end{proof}
The above statement is equivalent to there exists sequence $a_{n,k}$ such that $||\sum_{k=2}^{n}a_{n,k}\gamma_{k}-\gamma||_{H}$ converge to zero when $n$ goes to infinity.\\ 

Let $e_{i}$ be the sequence with 1 in the i-th entry, zero otherwise. Define $R_{i}:\mathbb{N}\rightarrow\mathbb{N}$, by sending $p$ to $p\, mod\, i$. Observe that for any finite $n$, the i-th component of $x_{n}-\gamma$ is periodic with period $L_{n}$, where $L_{n}$ is the least common multiple of numbers less than or equal to $n$. If there exists $a_{n,k}$ such that $\frac{|x_{n}-\gamma|_{i}^{2}}{i(i+1)}$ is smaller than any positive number for all $1\leq i\leq L_{n}-1$ for large enough $n$ then the Riemann hypothesis is true. In the following discussion, we only involve real numbers only, since if there are $a_{n,k}$ which are real sequences fulfill the conditions, $a_{n,k}$ is also a complex sequence as well. \\

Our approach is the following: For any finite $n\geq2$, we set up a system of inequalities $S(\epsilon,n)$:\\
\begin{center}
    $\sum_{k=2}^{n}a_{k}R_{k}(i)-1\leq\epsilon, i=1,...,L_{n}-1$\\
    $-\sum_{k=2}^{n}a_{k}R_{k}(i)+1\leq\epsilon, i=1,...,L_{n}-1.$\\
\end{center}
Where we re-scale $a_{n,k}$, rename it as $a_{k}$(as we have already specify $n$ in the system, we discard $n$ in the lower index) and $\epsilon$ is some positive number smaller than 1. Let $A$ be the coefficient matrix of the first set of inequalities, $a=(a_{2},a_{3},...,a_{n})^{T}$. The idea is to showing the existence of $a$ without explicitly constructing it.\\
We first show that the rank of $A$ is $n-1$.
\begin{theorem}
    The rank of $A$ is $n-1$.
\end{theorem}
\begin{proof}
    Consider $i\leq n-1$. by multiply first row and minus i-th row of $A$, we obtain a lower triangular sub-matrix in the form $k\lfloor\frac{i}{k}\rfloor$, where $k=2,...n$. so they are linear independent, so the rank of $A$ is $n-1$.
\end{proof}

Now we show a fact in linear algebra.

\begin{lemma}
    Given $A\in\mathbb{R}^{m\times n}$, $m>n$, $x,y\in\mathbb{R}^{n}$. If all the entries of $A$ is non-negative and each column of $A$ has at least one non-zero entries, and $x\geq y$ component-wise, $x\neq y$, then $Ax\geq Ay$ component-wise and $Ax\neq Ay$.
\end{lemma}
\begin{proof}
    We use induction to prove this. Let $k=1$, $A$ be $m$ by $k$ matrix, if $x \geq y$ with $x\neq y$, $Ax=(A_{1}, A_{2},...)^{T}x$, since all $A_{i}$ are non-negative and at least one entries of $A$ is not zero, clearly $A_{i}\geq A_{i}y$. Assume this hold for some positive integer $k$. Consider $A$ a $m$ by $k+1$ matrix, now $x_{i}\geq y_{i}$ for all $i$ with $x\neq y$. Let $A=(A', a)$, where $A'$ is a $m$ by $k$ matrix and $a$ is a $m$ by 1 matrix. All of them have non-negative entries. Let $x=(x',x_{k+1})^{T}$ and $y=(y', y_{k+1})^{T}$ similarly. Now $x\geq y$ and $x\neq y$, if all the cases that $x_{i}>y_{i}$ are in $x', y'$, then by assumption $A'x'\geq A'y'$ with $A'x'\neq A'y'$, and $ax_{k+1}\geq ay_{k+1}$ since $a$ is non-negative. If $x_{k+1}>y_{k+1}$, since $a$ has at least one nonzero entries $a_{i}$ and it is positive, $a_{i}x_{k+1}>a_{i}y_{k+1}$, so $Ax\geq Ay$ and $Ax\neq Ay$. By induction, this holds for all natural number $n$. Thus the conclusion.
    
\end{proof}

\begin{lemma}
    There exists $Av\geq0$ with $v\geq0$.
\end{lemma}
\begin{proof}
    Observe that the entries of $A$ is non-negative and each column and row contains positive value. If $v$ is a positive vector, then $Av$ must be a positive vector as well.
\end{proof}

Now we do the following. Let $v$ be the positive vector such that of $Av\geq0$, let $\delta>0$, $A^{+}$ be the Moore-Penrose inverse off $A$, the inequality $-\delta v\leq y-A^{+}c\leq\delta v$ is true for some $y$ with $c=(1,...,1)^{T}$. Now apply $A$ to this inequality, since $A$ preserve inequality by lemma 2.3, $Av$ is positive, we obtained $-\delta Av\leq Ay-AA^{+}c\leq\delta Av$. We have the following estimate.

\begin{theorem}
    Given $\epsilon>0$, there exists $y$ such that $|Ay-AA^{+}c|_{i}\leq\epsilon$ for all $i$. That is, $S(\epsilon,n)$ always have solution for a given $\epsilon>0$.
\end{theorem}
\begin{proof}
    By choosing small enough $\delta$ in above, the conclusion is easily seen.
\end{proof}

We shift our focus to $\gamma$. Let $P_{n}=A_{n}A^{+}_{n}$ where $A_{n}$ is $L_{n}-1\times n-1$ matrix $A$ defined above, and $A^{+}_{n}$ is its Moore-Penrose inverse, which take the first $n-1$ entries of vectors as input. By properties of Moore-Penrose inverse, $P_{n}$ is a orthogonal projection with rank $n-1$. Denote $l^{\infty}$ be the space of bounded sequences and $c_{0}$ be space of sequence converge to zero, with supremum norm be their norm. We first introduce the conditions for strong convergence of linear bounded operator between Banach space.

\begin{theorem}
    Let $(T_{n})$ be a sequence of linear bounded operator from Banach space $X$ to $Y$, then $T_{n}$ is strongly converge to a linear bounded operator $T$ if and only if:\\
    (1) $T_{n}x$ converges for $x$ belongs to dense subset of $X$, and,\\
    (2) $||T_{n}||<C$ for some $C>0$.
\end{theorem}
The proof the above theorem is a standard $\frac{\epsilon}{3}$ argument.\\

Since $P_{n}$ is sequence of projection operator in finite dimensional space, it is a sequence of bounded operator from $c_{0}$ to $c_{0}$. Denote $||.||_{\infty}$ the supremum norm, $||.||$ is the operator norm when the target space is endowed with supremum norm, $||.||_{2}$ is the operator norm when the target space is endowed with inner product norm. We show that $P_{n}$ strongly converge to identity operator $I$ in $c_{0}$.

\begin{theorem}
    The sequence of $P_{n}$ strongly converge to $I$ in $c_{0}$.
\end{theorem}
\begin{proof}
    Consider $||P_{n}||$, for all $n\in\mathbb{N}$, $||P_{n}||\leq||P_{n}||_{2}=1$, so it is uniformly bounded. Now the space $V=\cup_{n}\mathbb{R}^{n}$ is dense in $c_{0}$. Consider $x\in V$, $x$ is a vector with finitely nonzero terms, for some large enough $n$, $P_{n}x=P_{n+m}x$ for $m\geq1$, so $P_{n}x$ converge for all $x\in V$. By Theorem 2.6, $P_{n}$ strongly converge to some bounded operator $P:c_{0}\rightarrow c_{0}$.\\
    Now by expressing $P_{n}^{2}=P_{n}$, we can show that $P^{2}=P$. So $P$ is a projection in $c_{0}$. Since $P$ is bounded, image of $P$ is closed. Now image of $P_{n}$ is contained in $V$ with rising rank, so $V\subset Im(P)$, since $Im(P)$ is closed, $Im(P)=c_{0}$. So $P=I$.
\end{proof}
Now we state a condition for a sequence of bounded operator strongly converge to identity. The proof is finished by Martin Argerami.
\begin{theorem}
    Let $X$ is a Banach space and $K(X)$ be the closure of finite rank operator of $X$. Let $T_{n}$ be a bounded sequence of operators such that $T_{n}S$ converge to $S$ for all $S\in K(X)$, then $T_{n}$ strongly converge to identity.
\end{theorem}
\begin{proof}
    For any $x\in X$, there exists a rank one operator $T$ such that $Sx=x$(There exists linear functional $f$ such that $f(x)=1$ and let $Sy=f(y)x$). So $T_{n}x=T_{n}Sx$, so it converge to $Sx=x$, thus the conclusion.
\end{proof}

Now let $P_{n}: l^{\infty}\rightarrow l^{\infty}$, it can be done because every $P_{n}$ is finite rank and $c_{0}\subset l^{\infty}$. We show that $P_{n}$ strongly converge to $I$ in $l^{\infty}$.
\begin{theorem}
    $P_{n}$ strongly converge to $I$ in $l^{\infty}$.
\end{theorem}
\begin{proof}
    We first show that $P_{n}T$ converge to $T$ for $T$ is a finite rank operator. Consider $||P_{n}T-T||=sup_{||x||\leq1}(||(P_{n}T-T)x||_{\infty})$. Since $T$ is finite rank, $||P_{n}T-T||=sup_{||x||\leq1}||P_{n}x_{m}-x_{m}||_{\infty}$, where $x_{m}=Tx$ and it is a finite dimensional vector. Since $T$ is a compact operator and the closed unit ball $B_{1}$ is a bounded set, $T(B_{1})$ is a bounded set, denote the closure of this set to be $K$, $K$ is compact since it is closed, and bounded in a finite dimensional space(range of $T$ is finite dimension and $T$ is compact operator). Since $K$ is compact, $x_{m,n}$ contain a convergent subsequence such that it converge to some $x_{m}\in K$. Consider $||P_{n}x_{m,n}-x_{m,n_{k}}||_{\infty}$ , $x_{m,n_{k}}$ being the subsequence,
    \begin{center}
        $||P_{n}x_{m,n_{k}}-x_{m,n_{k}}||_{\infty}$\\
        $\leq||P_{n}x_{m,n_{k}}-P_{n}x_{m}||_{\infty}+||P_{n}x_{m}-x_{m}||_{\infty}+||x_{m,n_{k}}-x_{m}||_{\infty}$\\
        $\leq||P_{n}||||x_{m,n_{k}}-x_{m}||_{\infty}+||P_{n}x_{m}-x_{m}||_{\infty}+||x_{m,n_{k}}-x_{m}||_{\infty}$\\
    \end{center}
    Now $||P_{n}||\leq||P_{n}||_{2}=1$ since it is orthogonal projection, the first and third term can be made less than $\frac{\epsilon}{3}$ by large enough $n$, the second term can be made less than $\frac{\epsilon}{3}$ by Theorem 2.7. So the whole term is less than $\epsilon$, so $||P_{n}x_{m,n_{k}}-x_{m,n_{k}}||_{\infty}$ converge to zero. Observe that this fact holds for any converging subsequence $x_{m,n_{k}}$. Now the sequence $||P_{n}x_{m,n}-x_{m,n}||_{\infty}$ is bounded and therefore contain a converging subsequence $Q_{n_{k}}$ in $\mathbb{R}$. Assume $Q_{n_{k}}$ converge to some value $y\in\mathbb{R}$ other than zero. Since $Q_{n_{k}}=||P_{n_{k}}x_{m,n_{k}}-x_{m,n_{k}}||_{\infty}$, we can found a subsequence $x_{n_{k_{j}}}$ such that it converge and $Q_{n_{k_{j}}}$ converge to zero by above result, contradiction, so $||P_{n}x_{m,n}-x_{m,n}||_{\infty}$ converge to zero. Since $||P_{n}T-T||\leq||P_{n}x_{m,n}-x_{m,n}||_{\infty}$, so $\lim_{n}||P_{n}T-T||=0$.\\
    Now consider $T$ be any bounded operator which belongs to closure of finite rank operator. Consider $||P_{n}T-T||$, which is equal to $||P_{n}T+P_{n}T_{m}-P_{n}T_{m}-T||$ for some $T_{m}$ being a finite rank operator converge to $T$. By triangle inequality, we have 
    \begin{center}
    $||P_{n}T-T||\leq||P_{n}T-P_{n}T_{m}||+||P_{n}T_{m}-T||$,\\
    $\leq||P_{n}T-P_{n}T_{m}||+||P_{n}T_{m}+T_{m}-T_{m}-T||$\\
    $\leq||P_{n}||||T_{m}-T||+||P_{n}T_{m}-T_{m}||+||T_{m}-T||$ by triangle inequality,\\
    \end{center}
    Now $||P_{n}||\leq||P_{n}||_{2}=1$ since it is orthogonal projection, we can choose large enough $n,m$ such that $||T_{m}-T||<\frac{\epsilon}{3}$ and $||P_{n}T_{m}-T_{m}||<\frac{\epsilon}{3}$ since $P_{n}T_{m}$ converge $T_{m}$, so the above quantity is less than any positive number $\epsilon$ given large enough $n$, so $\lim_{n}P_{n}T=T$ for all $T\in K(X)$. Applying Theorem 2.8, $P_{n}$ strongly converge to identity in $l^{\infty}$.
\end{proof}

Now we can prove the Riemann hypothesis.\\
\textit{proof of Theorem 1.1}: Consider $\lim_{n}||x_{n}-\gamma||_{H}^{2}=\lim_{n}\sum_{i=1}^{\infty}\frac{|Aa_{n}-\gamma|_{i}^{2}}{i(i+1)}$, we have\\
\begin{center}
    $||x_{n}-\gamma||_{H}^{2}=\lim_{n}\sum_{i=1}^{\infty}\frac{|Aa_{n}-\gamma|_{i}^{2}}{i(i+1)}$\\
    $||x_{n}-\gamma||_{H}^{2}=\lim_{n}(\sum_{i\in J}\frac{|Aa_{n}+P_{n}\gamma-P_{n}\gamma-\gamma|_{i}^{2}}{i(i+1)}+\sum_{i=1}^{\infty}\frac{|1|}{iL_{n}(iL_{n}+1)})$ where $j\in J$ if $R_{L_{n}}(j)<L_{n}$\\
    $||x_{n}-\gamma||_{H}^{2}\leq\lim_{n}(\sum_{i\in J}\frac{(|Aa_{n}-P_{n}\gamma|_{i}+|P_{n}\gamma-\gamma|_{i})^{2}}{i(i+1)})+\lim_{n}\sum_{i=1}^{\infty}\frac{1}{iL_{n}(iL_{n}+1)}$ by triangle inequality.\\
\end{center}
Now by Theorem 2.5, given $\epsilon>0$, there always exists $a_{n}$ such that $|Aa_{n}-P_{n}\gamma|_{i}<\epsilon$ for all $n$ and all $i$, and by Theorem 2.9, $P_{n}$ strongly converge to $I$, so the first term goes to zero. For the second term, since $L_{n}\geq n$, thus $\frac{1}{iL_{n}(iL_{n}+1)}\leq\frac{1}{in(in+1)}$, which converges to zero. Apply dominated convergence theorem will give us zero as well. So the whole term converge to zero. By Theorem 2.1, the Riemann hypothesis is true.

\section*{Acknowledgement}
I thank Dr. Billy Leung, Mr. Pak Tik Fong, Mr. Dave Yeung, and Dr. Kenny Yip for their insightful feedback. I thank Martin Argerami for providing insight to the proof. I also thank Prof. Michel Balazard for pointing out the mistake in the original version.\\
This research did not receive any specific grant from funding agencies in the public, commercial, or not-for-profit sectors.

\bibliographystyle{abbrv}
\bibliography{ref}

\end{document}